\documentclass[11pt]{article}
\usepackage{amsmath,amsthm,amsfonts,amssymb,amscd, amsxtra}
\usepackage{color}
\usepackage{lscape}
\newtheorem{theorem}{Theorem}
\newtheorem{lemma}[theorem]{Lemma}

\newtheorem{corollary}[theorem]{Corollary}

\newtheorem{remark}{Remark}

\newtheorem{assumption}{Assumption}
\begin{document}
%%%%%%%%%%%%%%%%%%%%%%%%%%%%%%%%%%%%%%%%%%%%%%%%%%%%%%%%%%%%%
\title{On the global convergent of an inexact quasi-Newton conditional gradient method for constrained nonlinear systems}

\author{
    M.L.N. Gon\c calves
    \thanks{IME, Universidade Federal de Goi\'as, Goi\^ania, GO 74001-970, Brazil. (E-mails: {\tt
       maxlng@ufg.br} and {\tt fabriciaro@gmail.com}). The work of these authors was
    supported in part by CAPES, FAPEG/CNPq/PRONEM-201710267000532, and CNPq Grants 406975/2016-7 and  302666/2017-6.}
  \and F.R. Oliveira  \footnotemark[1]
}
 \maketitle
%%%%%%%%%%%%%%%%%%%%%%%%%%,,,,,,,,,,,,,,,,,,,,,,,,,,,,,,,
\begin{abstract}
In this paper, we propose a globally convergent method for solving constrained nonlinear systems.
The method combines an efficient Newton conditional gradient method with a derivative-free and nonmonotone linesearch strategy.
%At each iteration, we first solve a system of linear equations approximately, and then perform a specialized version of the conditional gradient method  in %order to get the Newton iteration back to the feasible set. Finally, the feasible iteration is accepted  if a nonmonotone   linesearch strategy is satisfied.
%followed by a linesearch strategy in which the acceptance of the trial step is tested.
The global convergence analysis of the proposed method is established under suitable conditions, and some  preliminary numerical experiments are given to illustrate its performance.
%Under suitable conditions, the global convergence analysis of the proposed method is established. Some preliminary numerical experiments are given to %illustrate the performance of the  method.
\end{abstract}


\maketitle
%%%%%%%%%%%%%%%%%%%%%%%%%%%%%%%%%%%%%%%%%%%%%%%%%%%%%%%%
\section{Introduction}\label{sec:int}
Let $F:\Omega \to \mathbb{R}^n $ be a continuously differentiable nonlinear function and $\Omega \subset \mathbb{R}^n$ be an open set.
Consider the problem of finding a vector $x\in \Omega$ such that
\begin{equation}\label{eq:p}
F(x) = 0.
\end{equation}
Among various methods for solving  unconstrained nonlinear system \eqref{eq:p}, the Newton method is regarded as one of the most effective. Basically, it
%It is well-known that the  Newton method applied to solve unconstrained nonlinear system \eqref{eq:p}
generates a sequence $\{x_k\}$ in such a way that
\[x_{k+1} = x_k + s_k, \quad \forall k\geq 0,\]
 where the Newton direction $s_k$ is computed by solving the linear system
\begin{equation}\label{system}
F'(x_k) s_k = - F(x_k).
\end{equation}
We refer the reader to \cite{Argyros2013,Birgin2003,MAX1,Goncalves2016}  where convergence results  of the Newton method and its variants have been discussed.

Consider now the constrained nonlinear system
\begin{equation}\label{eq:223}
F(x)=0, \quad x\in C,
\end{equation}
where $C \subset \Omega$ is a nonempty convex compact set. Various numerical methods for solving \eqref{eq:223} have been recently proposed and studied in the literature. Many of them are combinations of Newton methods with some strategies  taking into account the constraint set.
Strategies based on  projections, trust region,  active set and gradient methods have  been used; see, e.g.,
\cite{morini1, bellavia2006, echebest2012, Mangasarian1, sandra, cruz2014, MACCONI2009859, mariniquasi, marinez2, wang2016, Zhang1, Zhu2005343}.

A Newton conditional gradient (Newton-CondG) method  was proposed in \cite{CondG} (see \cite{Oliveira2017} for its inexact version) to compute approximate solutions of \eqref{eq:223}. Briefly speaking, the latter method consists of computing a Newton step  and later applying a conditional gradient (CondG) procedure in order to get the Newton iterative  back to the feasible set. In general, the CondG method and its variants require, at each iteration, to minimize a linear  function over the constraint set, which, in general, is significantly simpler than the projection step arising in many proximal-gradient methods. Moreover, depending on the application, linear optimization oracles may provide solutions
with specific characteristics leading to important properties such as sparsity and low-rank; see, e.g., \cite{Freund2014, ICML2013_jaggi13} for a discussion
on this subject.
As shown in \cite{CondG, Oliveira2017}, the Newton-CondG method as well as its inexact version performed well and compared favorably with other methods. However, no globalization strategy was considered in  \cite{CondG, Oliveira2017}  and hence only local convergence analyses of these methods were presented.

%However, the Newton-CondG methods in \cite{CondG, Oliveira2017} do not have any globalization strategy  and only local convergence analyses were %presented.

Therefore,  the aim of this article is to propose and analyze a version global of the method in \cite{Oliveira2017}.
It is worth pointing out that, in many cases, the strategy of globalization may become the methods more robustness.
Usually,  the global convergence of the methods for solving \eqref{eq:p} is obtained by ensuring   the decreasing of the merit function
\begin{equation}\label{eq:6783}
f(x) = \frac{1}{2} \|F(x)\|^2.
\end{equation}
See, for example, \cite{william2003, cruz2014, mariniquasi, marinez2, morini2016}.
However, for the inexact quasi-Newton method, the direction $s_k$, which is an approximate solution of  \eqref{system} with $F'(x_k)$ replaced by  an approximation of it,  may not be a descent direction of \eqref{eq:6783}. Hence, in this case, only nonmonotone globalization strategy can be considered.
%we must combine our inexact Newton-like conditional gradient method with  a  nonmonotone globalization strategy.
Almost all of these strategies are based on approximate norm descent condition  proposed in \cite{Fukushima2000}.
This condition can be described as follows: a sequence of feasible iterates $\{x_k\}$ is generated in such a way that the following nonmonotone condition is satisfied
\begin{equation}\label{desc}
\|F(x_{k+1})\| \leq (1+\eta_k) \|F(x_k)\|, \quad \forall k\geq 0,
\end{equation}
where $\{\eta_k\}$ is a positive sequence such that
\begin{equation}\label{condition}
\sum_{k = 0}^{\infty} \eta_k \leq \eta < \infty.
\end{equation}
Based on this condition, Morini proposed in \cite{morini2016} (see also \cite{mariniquasi}) a more general criterion, which replaced \eqref{desc} by the following inequalities:
\begin{equation}\label{desc1}
\|F(x_k +\pi(s_k, \lambda_k) )\| \leq (1 - \alpha (1 + \lambda_k)) \|F(x_k)\|,
\end{equation}
or
\begin{equation}\label{desc2}
\|F(x_k + \pi(s_k, \lambda_k) )\| \leq (1 + \eta_k - \alpha \lambda_k) \|F(x_k)\|,
\end{equation}
with $\eta_k$ as in \eqref{condition}, $\lambda_k \in (0,1]$, $\alpha \in (0,1),$ and $\pi(s_k, \lambda_k)$ is a suitable direction.
%This latter strategy will be consider in this paper.
%We mention that Morini's globalization strategy will be consider in our global scheme.
We mention that the global method to be proposed here is based on  the  latter globalization criterion.
 In order to illustrate the robustness and efficiency of the new method, we report some preliminary numerical experiments on a set of box-constrained nonlinear systems and compare its performance with the local FD-INL-CondG method in \cite{Oliveira2017} and the constrained dogleg method~\cite{Bellavia2012}.
%We mention that the  latter globalization criterion will be used to present our global version of the inexact Newton-like conditional gradient method.
%We mention that  the method to be proposed here  is based on  the latter globalization criterion.
%Morini's globalization strategy will be the main globalization tool for the inexact Newton-like conditional gradient method.
%In order to illustrate the performance of the new method, we report some preliminary numerical experiments on a set of box-constrained nonlinear %systems and compare its  performance with the constrained dogleg method~\cite{Bellavia2012}.

The paper is organized as follows. Section \ref{sec:condGmet}  presents the global inexact quasi-Newton conditional gradient  method as well as its analysis of global convergence. Some preliminary numerical experiments for the proposed method are reported  in Section \ref{NunEx}.
%Finally, some final remarks are presented in Section \ref{Finalrem}.
\\[2mm]
\noindent
{\bf Notation:} Throughout this paper, the Jacobian matrix of $F$ at $x\in \Omega$ is denoted by $F'(x)$. %whereas $\nabla f$ denotes the gradient of the merit function.
The inner product and its associated Euclidean norm in $\mathbb{R}^n$ be denoted by $\langle\cdot,\cdot\rangle$ and $\| \cdot \|$, respectively. The $i$-th component of a vector $x$ is indicated by $(x)_i$.% the entry $(i,j)$ of a matrix $A$ is denoted as $(A)_{ij}.$

%%%%%%%
\section{The algorithm and its global convergence}\label{sec:condGmet}
Our goal in this section is to present  as well as analyze  a new iterative method, namely  the  global inexact quasi-Newton conditional gradient (GIQN-CondG) method, for solving \eqref{eq:223}.

%This section describes and analyzes an iterative method with a strategy of globalization for solving constrained systems of nonlinear equations.
%For this purpose, we combine the inexact Newton-like conditional gradient method proposed in \cite{Oliveira2017} with a strategy of globalization %introduced in \cite{morini2016}, obtaining the globalized inexact Newton-like conditional gradient method (GINL-CondG).

\subsection{GIQN-CondG method}\label{description}

This  subsection describes the GIQN-CondG method, which  is obtained basically by combining the  inexact Newton-like conditional gradient method proposed in \cite{Oliveira2017} with  a strategy of globalization  similar to the one in \cite{morini2016}. As already mentioned, in many cases, the strategy of globalization may become the methods more robustness.

The GIQN-CondG method is formally described as follows.
\noindent
\\
\hrule
\noindent
\\
{\bf  GIQN-CondG method}\\
\hrule
\begin{description}
\item[(S.0)] (Initialization) Let $x_0\in C$, $\alpha$, $\sigma\in (0,1)$, $\eta_k$ satisfying  \eqref{condition} and $\{\theta_j\}\subset[0,\infty)$ be given, and set $k=0$.
\item[(S.1)](Termination criterion) If $F(x_k) = 0,$ then \textbf{stop}.
\item[(S.2)](Computation of the approximate quasi-Newton  direction) Choose an invertible approximation $M_k$ of $F'(x_k)$. For the residual $r_k \in \mathbb{R}^n$ compute a duple $(s_k,y_k)  \in \mathbb{R}^n \times \mathbb{R}^n$ such that
\begin{equation}\label{aa0}
M_k s_k=-F(x_k)+r_k, \quad y_{k}=x_k+s_k.
\end{equation}
\item[(S.3)] (CondG procedure) If $y_k\in C$, set $\tilde{s}_k=s_k$; otherwise, let
\begin{equation}\label{tildes}
\tilde{s}_k = \mbox{CondG} (y_k,x_k,\theta_k \|s_k\|^2) - x_k.
\end{equation}
\item[(S.4)](Backtracking process) Set $s_+ = \tilde s_k.$ If $\|\tilde{s}_k\|\neq 0$ set $s_- =  - \tilde{s}_k$ else $s_- =  - s_k$.
\begin{itemize}
\item[(S.4.1)] Set $\lambda = 1$.
\item[(S.4.2)] Repeat
\begin{itemize}
\item[(S.4.2.1)] If $\pi(s_k, \lambda) := \lambda s_+$ satisfies  \eqref{desc1}, go to {\bf (S.5)}.\\
Else if $\pi(s_k, \lambda) := \lambda s_-$ satisfies $x_k + \pi(s_k,\lambda) \in C$ and \eqref{desc1}, go to {\bf (S.5)}.
\item[(S.4.2.2)] If $\|s_+\|\neq 0$  and $\pi(s_k, \lambda) := \lambda s_+$ satisfies  \eqref{desc2}, go to {\bf (S.5)}.\\
Else if $\pi(s_k, \lambda) := \lambda s_-$ satisfies $x_k + \pi(s_k,\lambda) \in C$ and \eqref{desc2}, go to {\bf (S.5)}.
\item[(S.4.2.3)] Set $\lambda = \sigma \lambda.$
\end{itemize}
\end{itemize}
\item[(S.5)](Computation of new iterative) Set $ \lambda_k = \lambda$, $p_k = \pi(s_k,\lambda_k)$, $x_{k+1} = x_k + p_k.$
\item[(S.6)](Update) Set $k\leftarrow k+1,$ and go to {\bf(S.1)}.
\end{description}
\noindent
{\bf end}\\
\hrule
\noindent
\\
Let us now describe the CondG procedure.
\noindent
\\
\hrule
\noindent
\\
{ {\bf CondG procedure} $z=\mbox{CondG}(y,x,\varepsilon)$} \label{CGM}\\
\hrule
\begin{description}
\item[ P0.] Set $z_1=x$ and  $t=1$.
\item[ P1.] Use the linear optimization (LO) oracle to compute an optimal solution $u_t $ of
\begin{equation}\label{eq:epslon1}
g_{t}^*=\min_{u \in  C}\{\langle  z_t-y,u-z_t \rangle\}.
\end{equation}
\item[ P2.] If $ g^*_{t}\geq -\varepsilon $,  set $z=z_t$ and {\bf stop} the procedure; otherwise, compute $\alpha_t \in \, (0,1]$ and $z_{t+1}$ as

$$
{\alpha}_t: =\min\left\{1, \frac{-g^*_{t}}{\|u_t-z_t\|^2}  \right\}, \qquad  z_{t+1}=z_t+ \alpha_t(u_t-z_t).
$$
\item[ P3.] Set $t\gets t+1$, and go to {\bf P1}.
\end{description}
{\bf end procedure}\\
\hrule
\noindent
\\

\begin{remark}
%In this paper we propose a globally convergent method to solve \eqref{eq:223} which differs from the method presented in \cite{mariniquasi} in two %respects. First, in our approach we calculate an inexact projection by the CondG procedure, instead in \cite{mariniquasi}, two exact projections are %calculated in each iteration. Second, the system in \eqref{aa0} is solved with $r_k = 0$ for all $k\geq 0$ while in our method we solve it in an approximate %way.
i) There are different choices for, or way to build, the matrix  $M_k$ and the residual $r_k$ in (S.2), which originate variations of the GIQN-CondG method. For example, by taking $r_k = 0$ and $M_k = F'(x_k)$ (resp. $M_k = F'(x_0)$), we obtain a globalized version of the Newton (resp. modified Newton) conditional gradient method proposed in \cite{CondG} (resp. \cite{Oliveira2017}). We refer the reader to \cite{Bogle1990, Broyden1971, Schubert1970} for some  derivative-free approaches for building $M_k$.
ii) Note that,  the CondG procedure in (S.3) is used in order  to obtain  an approximate projection of the inexact quasi-Newton iteration $y_k$ to the feasible set~$C$, and as a consequence, a possible feasible direction $\tilde s_k$. More discussions of this specialized  CondG procedure can be found in  \cite[Remark 1]{Oliveira2017}.
iii) The Backtracking process given in (S.4) is well-defined, since its repeat-loop in (S.4.2) terminates in a
finite number of steps. Indeed, as $F$ is a continuous function and $\eta_k$ is a positive scalar for every $k$, then there exists a small enough scalar $\hat{\lambda} > 0$ such that the following inequality is satisfied
\begin{equation*}
(F(x_k + \lambda s))_i^2 \leq (1 + \eta_k - \alpha \lambda)^2 (F(x_k))_i^2,
\end{equation*}
for $\lambda \in (0, \hat{\lambda})$ and $i = 1, \ldots, n$. Consequently,  condition \eqref{desc2}  trivially holds.
Moreover, since  $s_-$ may not be a   feasible search  direction,  it is necessary to check  the  feasibility of the new iterate in this case.
iv) The GIQN-CondG method is closely related to the quasi-Newton method in  \cite{mariniquasi}. However, they differ mainly in two respects. First,
our approach computes an inexact projection by the CondG procedure, whereas the method in \cite{mariniquasi}  requires,   in each iteration,   two exact projections.  As already mentioned,  in many applications,   computing the projection step may be more difficult  than solving \eqref{eq:epslon1}.
Second, in  \cite{mariniquasi},  the linear  system  \eqref{aa0} is solved exactly (i.e.,  $r_k = 0$ for every $k\geq 0$),
which may be expensive and difficult for medium and large scale problems.
\end{remark}
%%%%%%%%%%%%%%%%%%%%%%%%%%%%%%%%%%%%%%%%%%%%%%%%%%%%%%%%%%%%%%%%%%%%

\subsection{Global convergence analysis}

In this subsection, we present global convergence results for the GIQN-CondG method. Specifically, we show that the sequence  $ \{\|F(x_k)\|\}$ is convergent and, under stronger assumptions, it converges to zero. Moreover, the global convergence of the  sequence $\{x_k\}$ is also established.

%establish cases in which $ \{\|F(x_k)\|\}$ converge to zero, i.e., the method is successful in solving \eqref{eq:p}. Let us also show that under %appropriate assumptions the generate sequence $\{x_k\}$ is convergent.

The following lemma guarantees that the approximate norm descent condition \eqref{desc} is satisfied for every $k$ and  establishes  some  upper bounds for $\|F(x_k)\|$.

\begin{lemma}\label{lemm4:art2}
Let $\{x_k\}$ and $\{\lambda_k\}$ be generated sequences by GIQN-CondG method.
\begin{itemize}
  \item[i)] For all $k \geq 0,$ condition \eqref{desc} holds and
  \begin{align} \label{eq1:art2}
  \|F(x_{k+1})\| &\leq e^{\eta} \|F(x_0)\|,\\
  \alpha \lambda_k \|F(x_k)\| &\leq (1 + \eta_k) \|F(x_k)\| - \|F(x_{k+1})\|. \label{eq2:art2}
  \end{align}
  \item[ii)] Let $\{k_m\}$, with $m\geq 1$ and $k_1 \geq 1$, be the indices of the iterates satisfying \eqref{desc1}, i.e.,
\begin{equation}\label{eq4:art2}
\|F(x_{k_m})\| \leq (1 - \alpha (1 + \lambda_{k_m - 1})) \|F(x_{k_m - 1})\|.
\end{equation}
Then,
 \begin{equation}\label{eq5:art2}
\|F(x_{k_m})\| \leq (1 - \alpha)^m \,e^{\eta} \,\|F(x_0)\|.
\end{equation}
\end{itemize}
\end{lemma}
\begin{proof}
See  proofs of  \cite[Theorem~4.2]{morini2016} and \cite[Lemma~3.1]{mariniquasi} for itens (i) and (ii), respectively.
%The proof of item~(i) follows as in \cite[Theorem~4.2]{morini2016}. and \cite[Lemma~3.1]{mariniquasi}.
\end{proof}

 The next lemma presents  a basic property of the CondG procedure, whose proof can be found in \cite[Lemma~4]{CondG}.
\begin{lemma}  \label{pr:condi}
%%%%%%%%%%%%
For any $y,\tilde y\in \mathbb{R}^n$, $x, \tilde x \in C$  and $\mu \geq 0$, we have
$$
\| \mbox{CondG}(y,x,\mu) - \mbox{CondG}(\tilde y,\tilde x,0)\|\leq  \|y-\tilde y\|+ \sqrt{2\mu}.
$$
\end{lemma}

The following assumption is needed in order to  investigate the global convergence of the sequences $\{x_k\}$ and $ \{\|F(x_k)\|\}$.

%In order to investigate the global convergence of the sequence $\{x_k\}$ and, moreover, to provide an estimates for the direction $p_k$ generated in %(S.5), the following assumption is nedeed.

\begin{assumption}\label{assum2}
Approximation $M_k$ of $F'(x_k)$  is invertible  for every $k\geq 0$. Moreover, assume that $M_k$ and   the residual  $r_k$ satisfy
\begin{equation}\label{con:qn}
\|{M_k}^{-1}\| \leq c_1, \qquad \|r_{k}\|\leq c_2\|F(x_{k})\|, \quad \forall k\geq 0,
\end{equation}
for some scalars $c_1>0$ and $c_2\geq 0$.
\end{assumption}

\begin{remark}\label{remark2:art2}  i) It is easy to see that the first equality in \eqref{aa0} and Assumption~\ref{assum2} imply
$$\|s_k\| \leq c_1(1+c_2) \|F(x_k)\|.$$
%ii) There are some ways to built   matrix $M_k$ such that the
%Assumption \eqref{assum2} trivially holds; for example, it can be chosen as a two-point approximation to the secant equation, see \cite{???}.
ii) See, for example, \cite{MR967848,Cruz2006} for more details in  how  to built matrices $M_k$ such that the  Assumption~\ref{assum2} trivially holds.
\end{remark}

Assumption~\ref{assum2} is essential  to provide estimaties for $\{\tilde s_k\}$ and $\{p_k\}$, which will be useful in the global analysis of GIQN-CondG method.

%The following lemma provides an estimate for $\{\tilde s_k\}$ and $\{p_k\}$, which will be used to prove the convergence of $\{x_k\}$.

\begin{lemma} \label{lema1:art2}
Let $\{x_k\}$, $\{\|F(x_k)\|\}$ and $\{\lambda_k\}$ be generated sequences by GIQN-CondG method. Assume that Assumption~\ref{assum2} holds and $\{\theta_k\} \subset [0, \beta^2 / 2]$ where $\beta \geq 0$.  Then, for every $k\geq 0$,
\begin{itemize}
\item [i)]  $\|\tilde{s}_k\| \leq c_1 (1 + \beta)(1+c_2) \|F(x_k)\|;$
\item [ii)] $\|p_k\|  \leq c_1(1 + \beta)(1+c_2) \lambda_k \|F(x_k)\|.$
\end{itemize}
\end{lemma}
\begin{proof}
i) First of all, if $\tilde s_k = s_k$, then from Remark~\ref{remark2:art2}(i) follows that
\begin{equation*}
\|\tilde s_k\| \leq c_1(1+c_2) \|F(x_k)\|,
\end{equation*}
which, combined with the fact that $\beta \geq 0$, implies the inequality of item~(i).
On the other hand, if $\tilde s_k = 0$, the desired  inequality trivially holds. Finally, let us consider the case where
\begin{equation*}
0\neq \tilde{s}_k = \mbox{CondG}(y_k,x_k,\theta_k \|s_k\|^2) - x_k.
\end{equation*}
Using the fact that $ \mbox{CondG}(x,x,0) = x$ for all $x \in C$, Lemma~\ref{pr:condi} and the second equality in \eqref{aa0}, we obtain
\begin{align*}\label{eq14:art2}
\|\tilde{s}_k\| = \|\mbox{CondG}(y_k,x_k,\theta_k \|s_k\|^2) - \mbox{CondG}(x_k,x_k,0)\|
 \leq \|y_k - x_k\| + \sqrt{2 \theta_k} \|s_k\|
 \leq (1 + \beta) \|s_k\|,
\end{align*}
where the last inequality follows from $\sqrt{2 \theta_k} \leq \beta$.
Hence, from Remark~\ref{remark2:art2}(i), we conclude the prove of the item.
\\[2mm]
ii) It follows from GIQN-CondG method that
\begin{equation*}\label{eq15:art2}
\|p_k\|= \|\pi(s_k,\lambda_k)\| = \lambda_k\|\tilde s_k\|,
\end{equation*}
which, combined with  item (i), proves the inequality of item (ii).
\end{proof}

The next theorem discusses the global convergence of the sequences  $\{\|F(x_k)\|\}$, $\{\lambda_k \|F(x_k)\|\}$ and $\{x_k\}$ as well as the
case in which  the GIQN-CondG method fails to solve \eqref{eq:223}.

%The next result, firstly establishes the convergence of sequences $\{\|F(x_k)\|\}$ and $\{\lambda_k \|F(x_k)\|\}$. Then, the first case in which $\{\|%F(x_k)\|\}$ converges to zero as well as case where the GINL-CondG method fails to solve \eqref{eq:p} are addressed. Lastly, on appropriate %assumptions, the convergence of $\{x_k\}$ is proved.

\begin{theorem}\label{theo1:art2}
Let $\{x_k\}$, $\{\|F(x_k)\|\}$ and $\{\lambda_k\}$ be generated sequences by GIQN-CondG method. Then,
\begin{itemize}
\item[i)] The sequence $\{\|F(x_k)\|\}$ is convergent;
\item[ii)] The sequence $\{\lambda_k \|F(x_k)\|\}$ is convergent and such that
\begin{equation}\label{eq8:art2}
\lim_{k\rightarrow\infty} \lambda_k \|F(x_k)\| = 0;
\end{equation}
\item[iii)] If \eqref{desc1} is satisfied for infinitely many $k$, then $\lim_{k\rightarrow\infty} \|F(x_k)\| = 0.$  Now,
if $\|F(x_k)\|\leq \|F(x_{k+1})\|$ for all $k$ sufficiently large, then  $\lim_{k\rightarrow\infty} \lambda_k = 0$ and $\lim_{k\rightarrow\infty} \|F(x_k)\| \neq 0$;
\item[iv)] If in addition Assumption \ref{assum2} holds and $\{\theta_k\} \subset [0, \beta^2 / 2]$ where $\beta \geq 0$, then the sequence $\{x_k\}$ is convergent.
\end{itemize}
\end{theorem}
\begin{proof} The proofs of the items (i), (ii) and (iii)  follows the same pattern as proofs of  items (i), (ii) and (iii) of  \cite[Theorem 3.2]{mariniquasi}.
 \\[2mm]
\textit{iv)}
Our goal is to prove that $\{x_k\}$ is a Cauchy sequence and hence it converges. Before, let us first prove that $\sum_{k=0}^{\infty} \lambda_k \|F(x_k)\|$ is a convergent series. It follows from  \eqref{eq2:art2} that
\begin{align*}
\sum_{k=0}^{\infty} \lambda_k \|F(x_k)\| &\leq \sum_{k=0}^{\infty} \left( \dfrac{(1 + \eta_k)}{\alpha} \|F(x_k)\| - \dfrac{1}{\alpha} \|F(x_{k+1})\|   \right)\\
& = \sum_{k=0}^{\infty}  \dfrac{1}{\alpha} (\|F(x_k)\| - \|F(x_{k+1})\|) + \sum_{k=0}^{\infty} \dfrac{\eta_k}{\alpha} \|F(x_{k+1})\|\\
&\leq \dfrac{1}{\alpha}  \|F(x_0)\|+ \sum_{k=0}^{\infty} \dfrac{\eta_k}{\alpha} \|F(x_{k+1})\|,
\end{align*}
which, combined with \eqref{condition} and \eqref{eq1:art2}, yields
\begin{align*}
\sum_{k=0}^{\infty} \lambda_k \|F(x_k)\|
\leq \frac{1}{\alpha} \|F(x_0)\| + \sum_{k=0}^{\infty} \dfrac{\eta_k}{\alpha} e^\eta \|F(x_0)\|
\leq \left(\frac{1}{\alpha} + \dfrac{\eta}{\alpha} e^\eta  \right)  \|F(x_0)\|.
\end{align*}
Since $\lambda_k \|F(x_k)\|$ is positive for every $k$, we conclude that $\sum_{k=0}^{\infty} \lambda_k \|F(x_k)\|$ is convergent.
Hence, from  Lemma~\ref{lema1:art2}, we obtain
\begin{equation*}
\sum_{k=0}^{\infty}\|p_k\| \leq c_1(1 + \beta)(1+c_2)  \sum_{k=0}^{\infty}\lambda_k\|F(x_k)\| < \infty.
\end{equation*}
On the other hand, let $m \geq l$ and consider
\begin{equation}\label{eq19:art2}
\|x_m - x_l\| = \|p_l + p_{l+1} + \ldots + p_{m-1}\| \leq \sum_{k = l}^{\infty} \|p_k\| =  \sum_{k = 0}^{\infty} \|p_k\|-\sum_{k = 0}^{l-1} \|p_k\|.
\end{equation}
Taking the limit in \eqref{eq19:art2} as $l$ goes to infinity, we have $\|x_m - x_l\|$ tends to zero. This implies that
for every $\varepsilon > 0$, there exists $l$ sufficiently large such that $\|x_m - x_l\|\leq \varepsilon$,  for all  $m \geq l$. Therefore, $\{x_k\}$ is a Cauchy sequence and the proof of the item is complete.
\end{proof}
For the last two results we will assume that the Jacobian $F'$ is Lipschitz continuous.
\begin{assumption}\label{assump3}
Assume that the Jacobian $F' $ of $F$  satisfies
$$\|F'(x) - F'(y)\| \leq L \|x - y\|, \quad \forall x, y \in C.$$
\end{assumption}
We now  prove that, under additional assumptions, the $\{\|F(x_k)\|\}$ converges to zero.
%$\tilde s_k$ is possibly an inexact Newton step, then $\{\|F(x_k)\|\}$ converges to zero, in other words, the GINL-CondG method is globally convergent.
%Agora, vamos provar que se $\tilde s_k$ é possivelmente um passo de Newton inexact, então $\{\|F(x_k)\|\}$ converge para zero, ou seja, o método é globalmente convergente.

\begin{theorem}\label{theo2:art2}
Let $\{x_k\}$ be sequence generated by GIQN-CondG method. Assume that Assumptions~\ref{assum2} and \ref{assump3} hold. If for all $k$ sufficiently large the step $\tilde{s}_k$ satisfies
\begin{equation}\label{eq9:art2}
\|F'(x_k) \tilde{s}_k + F(x_k) \| \leq \delta\|F(x_k)\|, \quad 0\leq  \delta < 1 - 3\alpha,
\end{equation}
then $\lim_{k\rightarrow \infty} \|F(x_k)\| = 0.$
\end{theorem}
\begin{proof}
It follows from Lemma~\ref{lema1:art2} that
\begin{equation}\label{eq16:art2}
\|\tilde{s}_k\| \leq c_1 (1 + \beta)(1+c_2) \|F(x_k)\|.
\end{equation}
Let us now prove that \eqref{desc1} holds for infinitely many $k$. Since $\|\tilde{s}_k\| \neq 0$ (see \eqref{eq9:art2}), we have
\begin{align*}
F(x_k + \lambda_k \tilde{s}_k) &= F(x_k) + \int_{0}^1 F'(x_k + t \lambda_k \tilde{s}_k) \lambda_k \tilde{s}_k ~dt\\
& = (1 - \lambda_k) F(x_k) + \lambda_k (F'(x_k)\tilde{s}_k + F(x_k)) + \int_{0}^1 (F'(x_k + t \lambda_k \tilde{s}_k)- F'(x_k) )\lambda_k \tilde{s}_k ~dt.
\end{align*}
Using \eqref{eq9:art2}, \eqref{eq16:art2} and Assumption~\ref{assump3}, we obtain
\begin{align*}
\|F(x_k + \lambda_k \tilde{s}_k)\|
& \leq (1 - \lambda_k) \|F(x_k)\| + \lambda_k \delta \|F(x_k)\| + \frac{L}{2} \lambda_k ^2 \|\tilde{s}_k\|^2\\
& \leq (1 - \lambda_k + \lambda_k \delta)\|F(x_k)\| + \frac{L}{2} [c_1 (1 + \beta)(1+ c_2)]^2 \lambda_k^2 \|F(x_k)\|^2,
\end{align*}
which, combined with the fact that  $\lambda_k \in (0,1]$, yields
\begin{align*}
\|F(x_k + \lambda_k \tilde{s}_k)\|
& \leq (1 - \lambda_k + \lambda_k \delta)\|F(x_k)\| + \frac{L}{2} [c_1 (1 + \beta)(1+ c_2)]^2 \lambda_k \|F(x_k)\|^2\\
& = \left(1 - \lambda_k + \lambda_k \delta + \frac{L}{2} [c_1 (1 + \beta)(1+ c_2)]^2 \lambda_k \|F(x_k)\|\right)\|F(x_k)\|.
\end{align*}
As consequence of \eqref{eq8:art2}, we conclude that there exists a $\bar k$ such that  $(L/2) [c_1 (1 + \beta)(1+ c_2)]^2 \lambda_k \|F(x_k)\| < \alpha$ for $k\geq \bar k$. Hence, condition \eqref{desc1}  holds for $k\geq \bar k$  if
$$1 - \lambda_k + \lambda_k \delta + \alpha  \leq 1 - \alpha(1+\lambda_k).$$
or, equivalently,
$$ \lambda_k(1-\alpha- \delta)  \geq  2\alpha.$$
Therefore, since \eqref{eq9:art2} implies  $0<2\alpha/(1-\alpha- \delta)< 1$, we conclude, from steps (S.4.1) and (S.4.2) of the  GIQN-CondG method, that
condition \eqref{desc1} holds for every $k\geq \bar k$, and hence the statement of the lemma trivially follows from Theorem~\ref{theo1:art2}(iii).
\end{proof}
Note that, the first equation in \eqref{aa0} and second inequality in \eqref{con:qn} imply that $\|M_k s_k+F(x_k)\|\leq c_2 \|F(x_k)\|$, for every $k\geq 0$. Hence,  condition \eqref{eq9:art2} trivially holds if   $\tilde s_k=s_k$ and $M_k=F'(x_k)$  for all $k$ sufficiently large,  and $c_2 < 1 - 3\alpha$. In the next corollary, we give conditions in which \eqref{eq9:art2} also holds when  $M_k$ is only an approximate of $F'(x_k)$.

\begin{corollary} Let $\{x_k\}$ be sequence generated by GIQN-CondG method. Assume that Assumptions~\ref{assum2} and \ref{assump3} hold. If for all $k$ sufficiently large the steps $s_k$ and $\tilde{s}_k$ satisfy $ \tilde s_k = s_k$ and
\begin{equation}\label{eq10:art2}
\|F'(x_k) M_k^{-1}\| \leq \rho,  \qquad \|I - F'(x_k) M_k^{-1}\| \leq \upsilon , \qquad \upsilon + \rho c_2 < 1 - 3\alpha,
\end{equation}
where $\rho>0$, $\upsilon \geq 0$ and $c_2$ is given in Assumption~\ref{assum2}, then $\lim_{k\rightarrow \infty} \|F(x_k)\| = 0.$
\end{corollary}
\begin{proof}
By the first equality in \eqref{aa0},  $ \tilde s_k= s_k$, \eqref{eq10:art2} and Assumption~\ref{assum2} follow that
\begin{align*}
\|F'(x_k)\tilde s_k + F(x_k)\| &= \|- F'(x_k)M_k^{-1}(F(x_k) - r_k)+ F(x_k)\| \\
&\leq \|(I - F'(x_k) M_k^{-1}) F(x_k)\| + \|F'(x_k)M_k^{-1} r_k\| \\
&\leq(\upsilon + \rho c_2) \|F(x_k)\|.
\end{align*}
Hence, the statement of the corollary now follows from  Theorem~\ref{theo2:art2} with $\delta=\upsilon + \rho c_2$.
\end{proof}

\section{Numerical experiments}\label{NunEx}

This section reports results of some preliminary  numerical experiments obtained by applying the GIQN-CondG method to solve $17$ test problems of the form \eqref{eq:p} with $C =\{x \in \mathbb{R}^n: l \leq x \leq u\}, $ where $l, u \in\mathbb{R}^{n}$, see Table~\ref{tab1}.
We tested the following variants of the GIQN-CondG method which differ in the way that the approximation matrices $M_k$'s  are built.
In the FD-GIQN-CondG method, the  matrices $M_k$'s were approximated by finite differences, whereas in the BSU-GIQN-CondG and BPU-GIQN-CondG methods, we used the Broyden-Schubert Update \cite{Broyden1971, Schubert1970}  and  the Bogle-Perkins Update  \cite{Bogle1990}, respectively.
For the latter two methods, we also used the strategy of periodically (i.e., $k = 0$ and mod$(k-1, 5) = 0$) approximating  by finite differences  the matrices $M_k$'s.
We compare the performance of above variants with the local FD-INL-CondG method \cite{Oliveira2017} and the constrained Dogleg solver (CoDoSol),
which is a MATLAB package based on the constrained Dogleg method  \cite{Bellavia2012}, and available on  the web site  {\it http://codosol.de.unifi.it}. In the latter two methods, the Jacobian matrices were approximated by finite differences. The parameters of the CoDoSol were selected as recommended by the authors, see \cite[Subsection~4.1]{Bellavia2012}. All  numerical results were obtain using MATLAB R2016a on a  2.5GHz Intel(R) i5 with 6GB of RAM and Windows 7 ultimate operation system.

For all methods, the starting points were defined as  $x_0(\gamma) = l + 0.2 \gamma (u - l)$, where $\gamma\geq 0$. %for problems with finite lower and upper bounds and $x_0(\gamma) = 10^{\gamma}(1, \ldots,1)^T$ for the others.
Moreover, we used the same overall termination condition  $\|F(x_k)\|_{\infty}\leq10^{-6}$,  and  a failure was  declared  if either no progress was detected or the total number of iterations exceeded $300$. In the variants of the GIQN-CondG method,  the initialization data  were  $\alpha = 10^{-4}$, $\sigma = 0.5$, $\eta_k = 0.99^k (100 + \|F(x_0)\|^2)$ and $\theta_k = 10^{-5},$ for every $ k \geq 0$, and
the linear  systems in \eqref{aa0} were solved by direct methods, i.e., $r_k = 0$  for all $k\geq 0 $. The CondG procedure
stopped when either the stopping criterion given in P2 is satisfied  or the maximum of $300$ iterations are performed.
Note that, in this application, subproblem  \eqref{eq:epslon1} has a closed-form solution, i.e.,  if $(z_t)_i-(y)_i\geq0$, then $(u_t)_i=(l)_i$; otherwise  $(u_t)_i=(u)_i$.
The parameters of the FD-INL-CondG method  were chosen as the corresponding one of its global version (i.e., GIQN-CondG method).

Tables~\ref{tab2} and \ref{tab3} display all numerical results obtained. The methods were compared on the total number of iterates (It), number of F-evaluation (Fe) and CPU time in seconds (Time). The symbol ``$*$" indicates a failure, whereas  $\|F\|_{\infty}$ and $\zeta(q)$ are the  infinity norm of $F$ at the final iterate $x_k$ and  $\zeta\cdot 10^q$, respectively. In Table~\ref{tab2}, the number of F-evaluations of the  FD-INL-CondG method was omitted in all cases, because  it is always  equal to the number of iterations plus one.
% In all cases,  we omitted in Table~\ref{tab2} the number of F-evaluations of the  FD-INL-CondG method, because  it is always  equal to the number of %iterations plus one.

From Table~\ref{tab2}, in terms of  amount of problems solved, we can see that the FD-GIQN-CondG method was  more robust than the FD-INL-CondG method and CoDoSol. This because the FD-GIQN-CondG method solved $47$ problems of a total of $51$, whereas the FD-INL-CondG method and CoDoSol sucessfully ended in $42$ problems. Regarding to the number of iterations, we observe that the FD-GIQN-CondG and FD-INL-CondG methods had similar performance and, in general, they required less iterations than CoDoSol. Similar efficiency can also be observed for the number of F-evaluations of the FD-GIQN-CondG method and CoDoSol. The CPU times of the three methods were practically the same.

Comparing the methods in which $F'$ is not evaluated at each iteration, we can observe, from Table~\ref{tab3}, that the BSU-GIQN-CondG and BPU-GIQN-CondG methods  were similar in terms of robustness and efficiency.
%, but the first method was more efficient than latter one in respect of  the numbers of iterations and F-evaluations  where both methods successfully %ended.
Note also that the slower convergence rates of the  BSU-GIQN-CondG and BPU-GIQN-CondG methods are compensated by their smaller CPU times per iteration.  Such a behavior is due to the fact that quasi-Newton approximations of $M_k's$ are computationally cheaper.
%In general, all methods had a similar behavior with regard to CPU time.

As a summary of the previous discussion, we can say that the GIQN-CondG method seems to be a robust and efficient tool for solving box-constrained systems of nonlinear equations.

\begin{table}[h]
\centering
\caption{Test problems }\label{tab1}
\vspace{0.5cm}
{\small
\begin{tabular}{|c|c|c|c|}
\hline
Problem & Name and source  & n & Box \\
\hline
Pb 1  & Effati-Grosan problem 2 \cite{athanassios2010}   & 2 & $[-10,10]$\\
Pb 2  & Reactor $R = 0.935$ \cite{athanassios2010} & 2  & $[0,5]$\\
Pb 3  & Merlet  \cite{athanassios2010}& 2 & $[0,2 \pi]$\\
Pb 4  & Brown's almost linear system \cite[14.1.5]{Jones:2000}& 5 & $[-2,2]$\\
Pb 5  & Countercurrent reactors 2 \cite[Problem 4.2]{Vlcek} & 8 & $[-100, 10]$\\
Pb 6  & Chemical reaction problem  \cite[Problem 5]{sandra} & 67  & $[-20,20]$\\
Pb 7  & Yamamutra  \cite{athanassios2010} & 100  & $[-100,100]$\\
Pb 8  & Extended Freudenstein-Roth \cite[Problem 4.11]{Vlcek} & 100 & $[-100, 100]$\\
Pb 9 & Tridiagonal system \cite[Problem 4.7]{Vlcek}  & 100  & $[-5,5]$\\
Pb 10 & Extended Wood \cite[Problem 4.17]{Vlcek}  & 100  & $[-5,5]$\\
Pb 11 & Singular Broyden\cite[Problem 4.6]{Vlcek} & 100 & $[-100,1]$\\
Pb 12 & Extended Powell singular \cite[Problem 4.12]{Vlcek} & 100 & $[-5,5]$\\
Pb 13 & Broyden tridiagonal \cite[Problem 30]{more1} & 500  & $[-100, 0]$\\
Pb 14 & Structured Jacobian \cite[Problem 3.19]{Vlcek} & 500  & $[-100, 0]$\\
Pb 15 & Brent \cite[Problem 4.20]{Vlcek} & 500  & $[-100,100]$\\
Pb 16 & Bratu \cite[Problem 4.24]{Vlcek} & 1024 & $[-100,1.5]$\\
Pb 17 & Trigonometric function \cite[Problem 8]{william2003}  & 2000  & [-50,150]\\
\hline
\end{tabular}
}
\end{table}

\begin{table}[h]
\centering
\caption{Performance of the FD-GIQN-CondG, FD-INL-CondG methods and CoDoSol}\label{tab2}
\vspace{0.5cm}
{\scriptsize
\begin{tabular}{|cc|ccc|cc|ccc|ccc|}
\hline
  &  &   \multicolumn{3}{|c|}{FD-GIQN-CondG}&  \multicolumn{2}{|c|}{FD-INL-CondG} & \multicolumn{3}{|c|}{CoDoSol}\\
\hline
 Problem& $\gamma$ & It & Fe & Time/$\|F\|_{\infty}$ &  It  & Time/$\|F\|_{\infty}$& It & Fe & Time/$\|F\|_{\infty}$\\
\hline
   %%%%%%%%%
 Pb  1 &1 & *    &     &                       &*     &                           & * &   &\\
       &2 & 29   & 31  & $7.55(-2)$/$6.74(-7)$ & *    &                           & 6 & 8 & $3.36(-2)$/$1.85(-8)$\\
       &3 & 28   & 30  & $3.05(-2)$/$6.22(-7)$ &19    & $7.84(-2)$/$2.01(-7)$     & 7 & 9 & $1.67(-2)$/$1.67(-10)$  \\
      %%%%%%%%%%
\hline
   %%%%%%%%%
 Pb  2 & 1 & 13  & 14  &$3.71(-1)$/$4.80(-12)$ & 13  &  $1.15(0)$/$4.80(-12)$     & 13 & 14 & $2.66(-1)$/$4.80(-12)$\\
       & 2 & 23  & 24  &$2.18(-2)$/$8.06(-8)$  & 23  &  $2.10(-2)$/$8.06(-8)$     & 23 & 24 & $4.41(-2)$/$8.06(-8)$\\
       & 3 & 33  & 34  &$1.52(-2)$/$7.31(-8)$  & 33  &  $2.26(-2)$/$7.31(-8)$     & 33 & 34 & $1.03(-2)$/$7.31(-8)$\\
      %%%%%%%%%%
\hline
  %%%%%%%%%
  Pb  3 &1 &3  & 4 &$5.13(-1)$/$1.41(-11)$  & 3  & $4.06(-1)$/$1.41(-11)$    & 3 & 4 & $1.79(-1)$/$1.99(-11)$\\
        &2 &4  & 5 &$1.15(-1)$/ $0$         & 3  & $1.04(-1)$/$7.35(-16)$    & 5 & 8 & $4.71(-2)$/$6.58(-12)$\\
        &3 &4  & 5 &$8.97(-3)$/ $7.35(-16)$ & 3  & $1.11(-2)$/$0.00(0)$      & 5 & 8 & $1.15(-2)$/$3.71(-12)$\\
  %%%%%%%%%%%%
\hline
   %%%%%%%%%
  Pb  4 &2.5 &*  &   &                         & *   &                           & 6 & 7 & $1.70(0)$/$3.58(-10)$\\
        &3.5 &8  & 9 &$2.37(-2)$/ $2.84(-11)$  & 11  & $3.50(-2)$/$4.21(-8)$     & 4 & 5 & $5.08(-2)$/$1.33(-10)$\\
        &4.5 &13 &14 &$4.19(-2)$/ $4.21(-8)$   & *   &                           & 6 & 8 & $1.39(-2)$/$7.10(-8)$\\
  %%%%%%%%%%%%
\hline
Pb    5 & 0 & 118 & 119 &$5.24(-1)$/ $2.02(-7)$ & *   &                 & *  &    & \\
        & 1 & *   &     &                       & *   &                 & *  &    & \\
        & 2 & *   &     &                       & *   &                 & *  &    & \\
  %%%%%%
\hline
Pb    6 & 0 & 18 & 19 &$2.84(0)$/ $3.37(-7)$ & *   &                         & *  &    & \\
        & 1 & 18 & 19 &$1.51(0)$/ $6.16(-8)$ & 25  & $1.96(0)$/$3.95(-7)$    & 17 & 18 & $1.65(0)$/$6.85(-9)$\\
        & 2 & 15 & 16 &$1.20(0)$/ $2.63(-8)$ & 14  & $1.11(0)$/$2.61(-7)$    & 16 & 19 & $1.34(0)$/$1.82(-7)$\\
  %%%%%%
  \hline
Pb    7 &1 &13  & 14  &$5.86(-1)$/ $1.20(-8)$ & 13  & $4.93(-1)$/$1.20(-8)$    & 16 & 17 & $4.26(-1)$/$3.52(-7)$\\
        &2 &10  & 11  &$1.12(-1)$/ $3.87(-7)$ & 10  & $1.49(-1)$/$3.87(-7)$    & 12 & 13 & $1.59(-1)$/$4.58(-7)$\\
        &3 &11  & 12  &$1.04(-1)$/ $1.28(-8)$ & 11  & $1.60(-1)$/$1.28(-8)$    & 13 & 14 & $1.49(-1)$/$1.09(-8)$\\
        %%%%%%%%%%
        \hline
  %%%%%%
Pb   8  &1 &34   & 35 &$4.85(-1)$/ $1.46(-13)$  & 34  & $5.17(-1)$/$1.42(-13)$    & * & &\\
        &2 &18   & 19 &$1.61(-1)$/ $5.15(-8)$   & 18  & $2.27(-1)$/$5.16(-8)$    & * & &\\
        &3 &9    & 10 &$7.25(-2)$/ $5.68(-14)$  & 9   & $1.17(-1)$/$5.68(-14)$   & 10 & 11 & $9.63(-2)$/$1.09(-10)$\\
        %%%%%%%%%%%%
        \hline
   %%%%%
Pb   9 &1     &13  & 14  &$1.99(-1)$/$7.05(-10)$ & 13  & $2.31(-1)$/$7.05(-10)$   & * & & \\
        &2    &9   & 10  &$2.27(-1)$/$7.65(-8)$  & 9   & $1.28(-1)$/$7.65(-8)$   & 10 & 12 & $1.01(-1)$/$1.21(-7)$\\
        &3.5  &7   &  8  &$5.32(-2)$/$4.04(-12)$ & 7   & $6.44(-2)$/$4.04(-12)$   & 7 & 8 & $7.18(-2)$/$2.74(-8)$\\
        %%%%%%%%%%%%
        \hline
   %%%%%
Pb   10  &1    &15  & 16  &$2.37(-1)$/ $8.03(-13)$ & 15  & $2.57(-1)$/$9.00(-13)$    & 21 & 28 & $3.48(-1)$/$7.91(-11)$\\
        &2    &5   & 6   &$5.48(-2)$/ $7.50(-8)$  & 5   & $6.79(-2)$/$7.50(-8)$    & 10 & 14 & $1.29(-1)$/$6.05(-13)$\\
        &3.5  &8   & 9   &$8.12(-2)$/ $2.92(-9)$  & 8   & $1.02(-1)$/$2.92(-9)$    &10 & 12 & $1.15(-1)$/$8.05(-11)$\\
        %%%%%%%%%%%%
        \hline
           %%%%%
Pb   11  &1  & 27  & 28  & $3.97(-1)$/$4.97(-7)$ & 27  & $3.00(-1)$/$4.97(-7)$   & 32 & 33 & $2.82(0)$/$8.16(-7)$\\
        &2  & 26  & 27  & $2.67(-1)$/$4.87(-7)$ & 26  & $3.09(-1)$/$4.87(-7)$   & 31 & 32 & $5.71(-1)$/$4.62(-7)$\\
        &3  & 25  & 26  & $2.22(-1)$/$2.66(-7)$ & 25  & $2.73(-1)$/$2.66(-7)$   & 29 & 30 & $4.16(-1)$/$6.92(-7)$\\
        %%%%%%%%%%%%
        \hline
           %%%%%
Pb   12 &1  &  14  & 15 & $3.31(-1)$/$9.06(-7)$ & 14  & $2.68(-1)$/$9.06(-7)$   & 15 & 16 & $4.10(-1)$/$9.85(-7)$\\
        &2 &  13  & 14 & $1.24(-1)$/$5.35(-7)$ & 13  & $1.65(-1)$/$5.35(-7)$   & 14 & 15 & $1.83(-1)$/$3.26(-7)$\\
        &3 &  13  & 14 & $1.16(-1)$/$5.35(-7)$ & 13  & $1.65(-1)$/$5.35(-7)$   & 14 & 15 & $1.59(-1)$/$3.26(-7)$\\
        %%%%%%%%%%%%
        \hline
           %%%%%
Pb   13 &1  & 10   & 11   & $1.64(0)$/$7.87(-8)$  & 10  & $1.56(0)$/$7.87(-8)$    & 16 & 17 & $3.00(0)$/$7.64(-8)$\\
        &2  & 10   & 11   & $1.45(0)$/$2.05(-10)$ & 10  & $1.52(0)$/$2.05(-10)$  & 16 & 17 & $2.65(0)$/$2.95(-13)$\\
        &3  & 9    & 10   & $1.27(0)$/$7.96(-8)$  & 9   & $1.37(0)$/$7.96(-8)$    & 15 & 16 & $2.44(0)$/$3.80(-12)$\\
        %%%%%%%%%%%%
        \hline
           %%%%%
Pb   14 &1  & 10  & 11  & $1.80(0)$/$5.57(-8)$  & 10  & $1.92(0)$/$5.57(-8)$     & 17 & 18 & $3.42(0)$/$2.74(-9)$\\
        &2  & 10  & 11  & $1.56(0)$/$9.72(-11)$ & 10  & $1.58(0)$/$9.72(-11)$   & 16 & 17 & $3.03(0)$/$3.12(-7)$\\
        &3  & 9   & 10  & $1.41(0)$/$6.35(-8)$  & 9   & $1.48(0)$/$6.35(-8)$     & 16 & 17 & $2.93(0)$/$9.17(-12)$\\
        %%%%%%%%%%%%
        \hline
           %%%%%
Pb   15 &1 & 15    & 16   & $2.06(0)$/$3.42(-8)$  & 15  & $2.33(0)$/$3.42(-8)$     & 19 & 20 & $2.87(0)$/$3.91(-7)$\\
        &2 & 13    & 14   & $1.62(0)$/$4.52(-7)$  & 13  & $1.99(0)$/$4.52(-7)$     & 16 & 17 & $2.22(0)$/$8.30(-8)$\\
        &3 & 11    & 12   & $1.36(0)$/$7.64(-11)$ & 11  & $1.65(0)$/$7.64(-11)$   & 13 & 14 & $1.82(0)$/$1.44(11)$\\
        %%%%%%%%%%%%
        \hline
           %%%%%
Pb   16 &1 & 1   & 2   & $9.91(-1)$/$9.52(-7)$ & 1  & $7.30(-1)$/$9.52(-7)$   & 10 & 11 & $7.28(0)$/$2.21(-7)$\\
        &2 & 2   & 3   & $1.19(0)$/$1.78(-8)$  & 2  & $1.21(0)$/$1.78(-8)$    & 10 & 11 & $7.16(0)$/$1.88(-7)$\\
        &3 & 1   & 2   & $6.58(-1)$/$9.52(-7)$ & 1  & $6.62(-1)$/$9.52(-7)$   &  9  & 10 & $6.09(0)$/$1.97(-7)$\\
        %%%%%%%%%%%%
        \hline
           %%%%%
Pb   17 &0 & 6    & 7    & $2.01(1)$/$6.40(-8)$  & *   &                        & * &     &\\
        &1 & 14   & 15   & $4.38(1)$/$3.70(-8)$  & 14  & $4.80(1)$/$3.01(-8)$   & 7 & 8   & $2.51(1)$/$5.75(-10)$\\
        &2 & 15   & 16   & $4.63(1)$/$2.11(-8)$  & 15  & $7.01(1)$/$1.35(-10)$  & 17 & 18 & $5.39(1)$/$1.49(-9)$\\
        %%%%%%%%%%%%
        \hline
\end{tabular}
}
\end{table}

\begin{table}[h]
\centering
\caption{Performance of the BSU-GIQN-CondG, BPU-GIQN-CondG methods}\label{tab3}
\vspace{0.5cm}
{\scriptsize
\begin{tabular}{|cc|ccc|ccc|}
\hline
  &  &   \multicolumn{3}{|c|}{BSU-GIQN-CondG}&  \multicolumn{3}{|c|}{BPU-GIQN-CondG} \\
\hline
 Problem& $\gamma$ &   It & Fe & Time/$\|F\|_{\infty}$& It & Fe & Time/$\|F\|_{\infty}$\\
\hline
   %%%%%%%%%
 Pb  1 &1 & 12 & 13 & $2.88(-1)$/$4.62(-9)$ & 32 & 33 & $2.24(-1)$/$2.95(-10)$ \\
       &2 & 36 & 38 & $3.20(-2)$/$6.40(-7)$ & 64 & 66 & $7.22(-2)$/$9.35(-7)$  \\
       &3 & 13 & 14 & $1.99(-2)$/$1.82(-9)$ & 12 & 13 & $1.24(-2)$/$1.02(-9)$   \\
      %%%%%%%%%%
\hline
   %%%%%%%%%
 Pb  2 & 1 & 17 & 18 & $1.63(-1)$/$3.62(-11)$ & 16 & 17 & $1.76(-1)$/$3.18(-7)$    \\
       & 2 & 31 & 32 & $1.37(-2)$/$1.97(-7)$  & 31 & 32 & $1.45(-2)$/$1.49(-9)$   \\
       & 3 & 46 & 47 & $1.33(-2)$/$3.37(-8)$  & 44 & 45 & $1.32(-2)$/$3.45(-8)$   \\
      %%%%%%%%%%
\hline
   %%%%%%%%%
  Pb  3 &1  &3 & 4 & $8.99(-2)$/$7.63(-7)$  & 3 & 4 & $8.70(-2)$/$7.63(-7)$   \\
        &2  &6 & 7 & $5.78(-2)$/$1.33(-10)$ & 6 & 7 & $3.85(-2)$/$1.33(-10)$  \\
        &3  &6 & 7 & $6.21(-3)$/$1.33(-10)$ & 6 & 7 & $6.88(-3)$/$1.33(-10)$  \\
  %%%%%%%%%%%%
\hline
Pb    4 &2.5 & 10 & 11 & $7.10(-1)$/$2.11(-7)$ & 10 & 11 & $1.54(-1)$/$2.03(-7)$  \\
        &3.5 & 10 & 11 & $1.88(-2)$/$2.41(-9)$ & 10 & 11 & $2.34(-2)$/$2.51(-9)$  \\
        &4.5 & 16 & 17 & $6.49(-2)$/$7.38(-7)$ & 15 & 16 & $8.47(-2)$/$5.75(-7)$   \\
        %%%%%%%%%%
        \hline
Pb    5 & 0 & 62 & 63 & $4.98(-1)$/$1.46(-9)$       & 104 & 105 & $3.63(-1)$/$2.64(-8)$  \\
        & 1 & *  &    &                             & 38  & 39  & $7.32(-2)$/$6.06(-10)$  \\
        & 2 & *  &    &                             & *   &     &    \\
        %%%%%%%%%%
        \hline
Pb    6 & 0 & * &  &                             & * &  &   \\
        & 1 & * &  &                             & * &  &   \\
        & 2 & * &  &                             & * &  &    \\
        %%%%%%%%%%
        \hline
Pb    7 &1 & 17 & 18 & $2.57(-1)$/$3.30(-9)$ & 22 & 23 & $5.49(-1)$/$7.15(-10)$  \\
        &2 & 13 & 14 & $6.05(-2)$/$4.75(-7)$ & 27 & 28 & $1.58(-1)$/$1.43(-11)$  \\
        &3 & 14 & 15 & $5.45(-2)$/$3.18(-7)$ & 35 & 36 & $2.02(-1)$/$6.31(-7)$   \\
        %%%%%%%%%%
        \hline
  %%%%%%
Pb   8  &1   & 28 & 29 & $3.01(-1)$/$6.68(-10)$ & 27 & 28 & $2.84(-1)$/$1.06(-11)$ \\
        &2   & 18 & 19 & $7.62(-2)$/$4.50(-10)$ & 78 & 79 & $2.48(-1)$/$2.10(-9)$\\
        &3   & 12 & 13 & $4.72(-2)$/$2.90(-12)$ & 12 & 13 & $4.69(-2)$/$4.83(-12)$ \\
        %%%%%%%%%%%%
        \hline
   %%%%%
Pb   9 &1      &61 & 62 & $3.38(-1)$/$2.57(-7)$  & 92 & 93 & $6.48(-1)$/$5.18(-7)$ \\
        &2     &* & &                           & 127 & 128 & $7.17(-1)$/$9.67(-7)$  \\
        &3.5   &17 & 18 & $8.00(-2)$/$1.17(-11)$ & 13 & 14 & $4.38(-2)$/$9.49(-10)$ \\
        %%%%%%%%%%%%
        \hline
   %%%%%
Pb   10  &1   &  * & &  & * & & \\
        &2   & 8  & 9  & $4.76(-2)$/$2.21(-8)$ & 7  & 8  & $5.59(-2)$/$8.53(-8)$ \\
        &3.5 & 13 & 14 & $5.76(-2)$/$2.87(-8)$ & 20 & 21 & $7.47(-2)$/$2.20(-8)$ \\
        %%%%%%%%%%%%
        \hline
           %%%%%
Pb   11  &1   & 36 & 37 & $3.81(-1)$/$6.40(-7)$ & 37 & 38 & $2.53(-1)$/$7.17(-7)$ \\
        &2   & 35 & 36 & $1.05(-1)$/$4.15(-7)$ & 37 & 38 & $1.22(-1)$/$3.07(-7)$ \\
        &3   & 32 & 33 & $1.06(-1)$/$9.08(-7)$ & 33 & 34 & $1.50(-1)$/$5.87(-7)$ \\
        %%%%%%%%%%%%
        \hline
           %%%%%
Pb   12 &1    &  19 & 20 & $2.66(-1)$/$5.02(-7)$ & 20 & 21 & $2.27(-1)$/$8.19(-7)$\\
        &2   &  17 & 18 & $8.20(-2)$/$4.65(-7)$ & 18 & 19 & $1.06(-1)$/$7.41(-7)$ \\
        &3   &  17 & 18 & $7.28(-2)$/$4.65(-7)$ & 18 & 19 & $7.01(-2)$/$7.41(-7)$ \\
        %%%%%%%%%%%%
        \hline
           %%%%%
Pb   13 &1   &  14 & 15 & $1.12(0)$/$9.38(-8)$  & 13 & 14 & $9.78(-1)$/$9.93(-8)$  \\
        &2   &  13 & 14 & $7.98(-1)$/$6.96(-7)$ & 12 & 13 & $6.93(-1)$/$7.98(-7)$  \\
        &3   &  13 & 14 & $7.25(-1)$/$3.52(-9)$ & 12 & 13 & $6.98(-1)$/$5.44(-10)$ \\
        %%%%%%%%%%%%
        \hline
           %%%%%
Pb   14 &1   &  14  & 15 & $1.08(0)$/$7.29(-7)$  & 13 & 14 & $1.11(0)$/$4.11(-7)$  \\
        &2   &  14  & 15 & $8.19(-1)$/$5.11(-8)$ & 13 & 14 & $8.21(-1)$/$8.74(-8)$ \\
        &3   &  13  & 14 & $7.76(-1)$/$3.73(-7)$ & 13 & 14 & $8.64(-1)$/$2.86(-9)$ \\
        %%%%%%%%%%%%
        \hline
           %%%%%
Pb   15 &1   & 22 & 23 & $1.15(0)$/$5.12(-8)$  & 20 & 21 & $9.92(-1)$/$1.08(-7)$ \\
        &2   & 19 & 20 & $8.16(-1)$/$6.49(-7)$ & 17 & 18 & $8.02(-1)$/$9.29(-7)$ \\
        &3   & 15 & 16 & $6.17(-1)$/$9.11(-8)$ & 14 & 15 & $6.10(-1)$/$9.81(-9)$ \\
        %%%%%%%%%%%%
        \hline
           %%%%%
Pb   16 &1   & 1 & 2 & $8.67(-1)$/$9.52(-7)$ & 1 & 2 & $8.72(-1)$/$9.52(-7)$ \\
        &2   & 2 & 3 & $1.15(0)$/$1.78(-8)$  & 2 & 3 & $1.18(0)$/$1.78(-8)$  \\
        &3   & 1 & 2 & $6.14(-1)$/$9.52(-7)$ & 1 & 2 & $6.20(-1)$/$9.52(-7)$ \\
        %%%%%%%%%%%%
        \hline
           %%%%%
Pb   17 &0   &  8  & 9  &$1.31(1)$/$4.52(-10)$  & 11 & 12 & $1.47(1)$/$1.81(-7)$ \\
        &1   &  18 & 19 &$2.27(1)$/$9.97(-7)$   & 25 & 26 & $3.05(1)$/$1.50(-7)$ \\
        &2   &   * &    &  & * &  &  \\
        %%%%%%%%%%%%
        \hline
\end{tabular}
}
\end{table}

\end{document}